\documentclass[11pt]{article}
\usepackage{exscale,relsize}
\usepackage{amsmath}
\usepackage{amsfonts}
\usepackage{amssymb}
\usepackage{calc}
\usepackage{theorem}
\usepackage{pifont}      
\newcommand{\im}{\mathrm{Im}}
\newcommand{\dom}{\mathrm{dom}}

\newcommand{\RX}{\ensuremath{\,(-\infty,+\infty]}}

\newcommand{\RR}{\ensuremath{\mathbb R}}
\newcommand{\cA}{\mathcal{A}}

\newtheorem{theorem}{Theorem}[section]
\newtheorem{definition}[theorem]{Definition}

\newtheorem{example}[theorem]{Example}
\theoremstyle{plain}{\theorembodyfont{\rmfamily}}
\newtheorem{proposition}[theorem]{Proposition}
\newtheorem{lemma}[theorem]{Lemma}

\theoremstyle{plain}{\theorembodyfont{\rmfamily}

\theoremstyle{plain}{\theorembodyfont{\rmfamily}

\begin{document}
\title{Optimal Pricing for Optimal Transport}

\author{
Sedi Bartz
\ \ and\ \ \ 
Simeon Reich
\thanks{E-mail addresses: \texttt{bartz@techunix.technion.ac.il} (S. Bartz),\ \texttt{sreich@techunix.technion.ac.il} (S. Reich).}
\\
\small{\textit{Department of Mathematics,
The Technion -- Israel Institute of Technology,
32000 Haifa, Israel.}}
}

\date{January 13, 2014}

\maketitle


\begin{abstract} \noindent
Suppose that $c(x,y)$ is the cost of transporting a unit of mass from $x\in X$ to $y\in Y$ and suppose that a mass distribution $\mu$ on $X$ is transported optimally (so that the total cost of transportation is minimal) to the mass distribution $\nu$ on $Y$. Then, roughly speaking, the Kantorovich duality theorem asserts that there is a price $f(x)$ for a unit of mass sold (say by the producer to the distributor) at $x$ and a price $g(y)$ for a unit of mass sold (say by the distributor to the end consumer) at $y$ such that for any $x\in X$ and $y\in Y$, the price difference $g(y)-f(x)$ is not greater than the cost of transportation $c(x,y)$ and such that there is equality $g(y)-f(x)=c(x,y)$ if indeed a nonzero mass was transported (via the optimal transportation plan) from $x$ to $y$. We consider the following optimal pricing problem: suppose that a new pricing policy is to be determined while keeping a part of the optimal transportation plan fixed and, in addition, some prices at the sources of this part are also kept fixed. From the producers' side, what would then be the highest compatible pricing policy possible? From the consumers' side, what would then be the lowest compatible pricing policy possible?
We have recently introduced and studied settings in $c$-convexity theory which gave rise to families of $c$-convex $c$-antiderivatives, and, in particular, we established the existence of optimal $c$-convex $c$-antiderivatives and explicit constructions of these optimizers were presented. In applications, it has turned out that this is a unifying language for phenomena in analysis which used to be considered quite apart. In the present paper we employ optimal $c$-convex $c$-antiderivatives and conclude that these are natural solutions to the optimal pricing problems mentioned above. This type of problems drew attention in the past and existence results were previously established in the case where $X=Y={\RR}^n$ under various specifications. We solve the above problem for general spaces $X,Y$ and real-valued, lower semicontinuous cost functions $c$. Furthermore, an explicit construction of solutions to the general problem is presented. 
\end{abstract}



\section{Introduction and Preliminaries}
Recall the Monge-Kantorovich minimization problem: let $(X,\mu)$ and $(Y,\nu)$ be two probability spaces. We denote by $\Pi(\mu,\nu)$ the set of all probability measures $\pi$ on the product set $X\times Y$ such that the marginals of $\pi$ are $\mu$ and $\nu$. That is, for all measurable sets $A\subset X$ and $B\subset Y$, we have $\pi(A\times Y)=\mu(A)$ and $\pi(X\times B)=\nu(B)$. Given a function $c:X\times Y\to\RR$, we seek a minimizer of
\begin{equation}
C(\mu,\nu):=\inf_{\begin{array}{c}\pi\in\Pi(\mu,\nu) \end{array}}\int_{X\times Y} c(x,y)d\pi(x,y).
\end{equation}     
The measures $\pi\in\Pi(\mu,\nu)$ are called $transport\ plans$ or \emph{transference\ plans}. The measures $\pi\in\Pi(\mu,\nu)$ achieving the infimum are called $optimal\ transport\ plans$. The classical interpretation of this problem is the problem of minimizing the $total\ cost\ C(\mu,\nu)$ of transporting the mass distribution $\mu$ to the mass distribution $\nu$, where the $cost$ of transporting one unit of mass at the point $x\in X$ to one unit of mass at the point $y\in Y$ is given by the cost function $c(x,y)$. 

Our discussion here is possible because of the role $c$-convexity plays in the theory of optimal transport, more precisely, in the Kantorovich duality theorem. The  theory of $c$-convexity is a main branch of \emph{abstract convex analysis} (or \emph{generalized convexity}) and has attracted more and more attention in recent years. Now this topic is studied continually, both theoretically and from the point of view of applications. A unifying and very detailed treatment of abstract convex analysis and, in particular, of $c$-convexity can be found in \cite{sin}. We now recall the basic definitions and notations of $c$-convexity theory: unless otherwise specified, throughout the paper $X$ and $Y$ are arbitrary  sets and $c:X\times Y\to\RR$ is an arbitrary function. We say that a function $f:X\to\RX$ is proper if $\dom (f):=\{x\in X\ |\ f(x)<\infty\}$ is not empty.

\begin{definition}[c-transform]
Let $X$ and $Y$ be nonempty sets and let $c:X\times Y\to\mathbb{R}$ be a
function. Given
a function $f:X\to[-\infty,\infty]$, its $c$-transform $f^c:Y\to[-\infty,\infty]$ is defined by
\begin{equation}
f^c(y):=\sup_{x\in X}\ c(x,y)-f(x),\ \ \ y\in Y.
\end{equation}
Similarly, the $c$-transform of a function $g:Y\to[-\infty,\infty]$ is the function $g^c:X\to[-\infty,\infty]$ defined by
\begin{equation}
g^c(x):=\sup_{y\in Y}\ c(x,y)-g(y),\ \ \ x\in X.
\end{equation}
\end{definition}

\begin{definition}[c-convexity]\label{cconvexity}
A proper function $f:X\to\RX$ is said to be $c$-convex if there exists a (necessarily proper) function $g:Y\to\RX$ such that $f=g^c$. The set of all $c$-convex functions defined on $X$ is denoted by $\Gamma_c(X)$. The function $(f^c)^c$ is the $c$-convexification of $f$ and is denoted by $f^{cc}$.
\end{definition}

The $c$-transform of a function is also known as its \emph{c-conjugate} function. This generalization of Fenchel's conjugate function from classical convex analysis was introduced and studied by Moreau in \cite{mor}. Sometimes, the function $c$ is allowed to take the values $\pm\infty$; however, in our discussion we focus our attention on the settings in the above definitions. When referring to the theoretical study of $c$-convexity, the function $c$ is called a \emph{coupling function} between $X$ and $Y$, while in the particular application to the study of optimal transport it is the \emph{cost function}. As is often standard in convex analysis, the indicator function of a subset $S$ of $X$ is the function $\iota_S:X\to\RX$ defined by
$$
\iota_{S}(x):=\Big\{\begin{array}{c}
                     0\ \ \ x\in S \\
                     \infty\ \ x\notin S.
                   \end{array}
$$
For example, for every $y\in Y$, the function $c(\cdot,y):X\to\RR$ is $c$-convex
since $c(\cdot,y)=\iota_{\{y\}}^c$. (In fact, the functions $c(\cdot,y),\ y\in Y$, are the ones playing the role linear functionals play in classical convex analysis.)

Clearly, for any function $f:X\to\RX$,
\begin{equation}
c(x,y)\leq f(x)+f^c(y)\ \ \ \mathrm{for\ all}\ x\in X\ \mathrm{and}\ y\in Y,
\end{equation}
which is a generalization of the Young-Fenchel inequality from classical convex analysis. The case of equality is captured in the following definition of the $c$-subdifferential. We denote the graph of a multivalued mapping $M:X\rightrightarrows Y$ by $G(M):=\{(x,y)|\ y\in M(x)\}$. The mapping $M$ is called proper if $\dom(M):=\{ x\in X |\ M(x)\neq\emptyset\}$ is not empty. The image of the mapping $M$ is the subset of $Y$ given by $\im(M):=\cup_{x\in X} M(x)$. 

\begin{definition}[c-subdifferential and c-antiderivative]\label{ctrans}
Let $f:X\to\RX$ be a proper function. The $c$-subdifferential of $f$ is the mapping
$\partial_c f:X\rightrightarrows Y$ defined by
\begin{align}
\partial_c f(x):=\ &\big\{y\in Y\ |\ f(x)+c(x',y)\leq f(x')+c(x,y)\ \ \forall x'\in X\big\}\\
\nonumber\\
=\ &\big\{y\in Y\ |\ f(x)+f^c(y)=c(x,y)\big\}.
\end{align}
When $\partial_c f(x)\neq\emptyset$, we say that $f$ is $c$-subdifferentiable at $x$. When $M:X\rightrightarrows Y$ and $G(M)\subset G(\partial_c f)$, we say that $f$ is a $c$-antiderivative of $M$.
\end{definition}

With these definitions at hand, we now recall the Kantorovich duality theorem. In fact, we borrow from Villani's monograph \cite{vil} a specific version of the theorem which, in addition, tells the story in terms of $c$-convexity:

\begin{theorem}\label{Kantorovich}
Let $(X,\mu)$ and $(Y,\nu)$ be two Polish probability measure spaces. Suppose that the cost function $c:X\times Y\to\RR$ is lower semicontinous and suppose further that it majorizes a function $a(x)+b(y),\ (x,y)\in X\times Y$, for some upper semicontinuous functions $a\in L_1(\mu)$ and $b\in L_1(\nu)$.  Then the following duality holds:
\begin{align}
\min_{\begin{array}{c} \pi\in\Pi(\mu,\nu) \end{array}}
\int_{X\times Y} & c(x,y)d\pi(x,y)\nonumber\\
= & \sup_{\begin{array}{c}f\in L_1(\mu),\ g\in L_1(\nu)\\
g-f\leq c\end{array}}\Bigg(\int_Y g(y)d\nu(y)-\int_X f(x)d\mu(x)\Bigg)\nonumber\\
= &\ \sup_{\begin{array}{c}f\in L_1(\mu)\end{array}}\Bigg(-\int_Y f^{-c}(y)d\nu(y)-\int_X f(x)d\mu(x)\Bigg)\nonumber\\
= &\ \sup_{\begin{array}{c} g\in L_1(\nu)\end{array}}\Bigg(\int_Y g(y)d\nu(y)-\int_X (-g)^{-c}(x)d\mu(x)\Bigg),\nonumber  
\end{align}
and in the above suprema one might as well assume that $f$ and $-g$ are $-c$-convex. Furthermore, if the optimal total cost $C(\mu,\nu)$ is finite, then for a transference plan $\pi\in\Pi(\mu,\nu)$, the following assertions are equivalent:\\
\ \\
(a)\ $\pi$ is an optimal transport plan;\\
(b)\ $\pi$ is concentrated on a $-c$-cyclically monotone set in $X\times Y$;\\ 
(c)\ There is a $-c$-convex function $f:X\to\RR$ such that $\pi$ is concentrated on the set where the equality $f^{-c}+f=-c$ holds. That is, $\pi$ is concentrated on $G(\partial_{-c} f)$. 
\end{theorem}
 
This is just part of the presentation of the Kantorovich duality theorem one finds in \cite{vil}, where a most extensive and detailed study of optimal transport is presented. We will comment further on the benefits of this presentation of the theorem in Section 3 below. Explicit relations between the Kantorovich duality and $c$-convexity, in particular, $c$-cyclic monotonicity (see Definition \ref{cycmondef} below), appeared as early as \cite{Bre1,Bre2}, if not earlier. A remark regarding the sign conventions we employ is now in order. In \cite{vil} the sign conventions of $c$-convexity were modified so that the Young-Fenchel inequality looks like this: $f^c-f\leq c$. This way the conventions fit better the Kantorovich duality, as can be seen in the first supremum above. However, this way the symmetry between $X$ and $Y$ is lost. In \cite{vil} the definition of the $c$-transform was modified; $c$-convex functions were considered on $X$ while on $Y$, $c$-concave functions were employed. We will employ our standard conventions as defined above. These conventions are compatible with classical convex analysis. However, as we can see in the lower two suprema above, this affects the way we state the duality. Our sign conventions, which will have a higher price when we deal with optimal transport theory issues, but will have a lower price when we deal with pure $c$-convexity issues, stem from the ``standard'' Young-Fenchel inequality with respect to $-c$, that is, $f^{-c}+f\geq -c$, which implies the inequality $(-f^{-c})-f\leq c$, and further, the inequality $-g+(-g)^{-c}\geq -c$, which implies in its turn that  $g-(-g)^{-c}\leq c$, as required. This problem with signs is rooted in the classical definition of $c$-cyclic monotonicity (Definition \ref{cycmondef} below) and in the standard definition of an optimal transport plan, which minimizes the total transport cost rather than maximizes it. In \cite{vil} the definition of $c$-cyclic monotonicity was also taken with an opposite sign to ours. Ours is the extension of the classical definition; see, for example, \cite{roc}.

The rest of the paper is organized as follows: in Section 2 we recall recent, necessary results from $c$-convexity theory regarding families of $c$-convex $c$-antiderivatives and, in particular, optimal $c$-convex $c$-antiderivatives, which were presented by the authors in \cite{BR2}. The main section of this paper is Section 3, where we conduct our more detailed discussion regarding optimal pricing and embed it in the general mathematical framework of the Kantorovich duality theorem as presented above. In order to present a class of specific examples and in order to shed some additional light on optimal prices, we recall in Section 4 concrete examples which were presented by the authors in \cite{BR2}. There, the coupling function is a general metric, we are able to avoid measure theoretic issues completely, and the optimal prices are interpreted as optimal constrained Lipschitz extensions.

\section{The Family $\cA_{[c,f|_s,M]}$ of $c$-Convex $c$-Antiderivatives and Its Envelopes}

In \cite{BR2} the authors of the present paper introduced and studied families of $c$-convex $c$-antiderivatives, defined as families of solutions to the following problem:

\begin{definition}
Given a mapping $M:X\rightrightarrows Y$, a $c$-antiderivative $f$ of $M$ and a subset $S$ of $\dom (M)$, we denote the set of all $c$-convex functions $h:X\to\RX$ which satisfy
\begin{equation}
G(M)\subset G(\partial_c h)\ \ \ \ and \ \ \ h|_S=f|_S
\end{equation}
by $\cA_{[c,f|_S,M]}$.
\end{definition}

Given a mapping $M:X\rightrightarrows Y$, recall that its inverse mapping $M^{-1}:Y\rightrightarrows X$ is defined by $M^{-1}(y)=\{x\in X|\ y\in M(x)\},\ \ y\in Y$. In the above setting, since $M(S)$ is a subset of $\dom (M^{-1})$, it is also possible to consider the $c$-dual problem: $f^c$ is a $c$-antiderivative of $M^{-1}$. We therefore denote the set of $c$-convex solutions $h:Y\to\RX$ of the problem
\begin{equation}
G(M^{-1})\subset G(\partial_c h)\ \ \ \ and \ \ \ h|_{M(S)}=f^c|_{M(S)}
\end{equation}
by $\cA_{[c,f^c|_{M(S)},M^{-1}]}$.

Besides the existence of solutions, the two main $c$-convexity theoretical results from \cite{BR2}, Theorem \ref{main} and Theorem \ref{mainformula} below, imply rather natural structure and duality relations between the families and, in particular, the existence of optimal solutions and duality relations between them. In our applications so far, the fact that we do not assume that $f$ is $c$-convex, just a $c$-antiderivative, is crucial. This means that we are given a function $f$ with good $c$-convexity properties only on $\dom(M)$, which can be arbitrary, and then $f|_S$ is extensible to a ``good'' solution on all of $X$. Indeed, this is crucial in our application to optimal transport in the next section as well. The nonemptiness, duality relations, and existence and duality relations for the envelopes are concentrated in the following result, the first one of the two:

\begin{theorem}\label{main}
Suppose that $f:X\to\RX$ is a $c$-antiderivative of the mapping $M:X\rightrightarrows Y$.  Suppose further that $\emptyset\neq S\subset\mathrm{dom} (M)$. Then $\cA_{[c,f|_S,M]}$ is nonempty and contains both its upper envelope, that is, the function $\gamma_{[c,f|_S,M]}:X\to\RX$ defined by
$$
\gamma_{[c,f|_S,M]}(x):=\sup\{h(x)\ |\ h\in\cA_{[c,f|_S,M]}\},
$$
as well as its lower envelope, that is, the function $\alpha_{[c,f|_S,M]}:X\to\RX$ defined by
$$
\alpha_{[c,f|_S,M]}(x):=\inf\{h(x)\ |\ h\in\cA_{[c,f|_S,M]}\}.
$$
In fact, if $h:X\to\RX$ is any function such that
\begin{equation}\label{problem}
G(M)\subset G(\partial_c h)\ \ \ \ and \ \ \ h|_S=f|_S,
\end{equation}
 then $\alpha_{[c,f|_S,M]}\leq h$ and $h^c\in\cA_{[c,f^c|_{M(S)},M^{-1}]}$. If $h$ is $c$-convex, then
\begin{equation}\label{conjuA}
h\in\cA_{[c,f|_S,M]}\ \ \Leftrightarrow\ \ \ h^c\in\cA_{[c,f^c|_{M(S)},M^{-1}]}.
\end{equation}
Furthermore,
\begin{equation}\label{conjugamma}
\alpha_{[c,f|_S,M]}^c=\gamma_{[c,f^c|_{M(S)},M^{-1}]}\ \ \ \ and\ \ \ \
\gamma_{[c,f|_S,M]}^c=\alpha_{[c,f^c|_{M(S)},M^{-1}]}.
\end{equation}
\\
In the case where  $S=\mathrm{dom}(M)$, we have
\begin{align}
&\gamma_{[c,f|_{\dom (M)},M]}=(f+\iota_{\dom (M)})^{cc}\label{gammafulldom}\\
and\ \ \ \ \ &\nonumber\\
&\alpha_{[c,f|_{\mathrm{dom}(M)},M]}(x)=\ (f^c+\iota_{\mathrm{Im}(M)})^c(x)\\
&\ \ \ \ \ \ \ \ \ \ \ \ \ \ \ \ \ \ \ \ =\ \sup_{(s,t)\in G(M)}\ [f(s)+c(x,t)-c(s,t)],\ \ \ \ x\in X.\label{conjufulldom}
\end{align}
In this case, if $h:X\to\RX$ is $c$-convex, then
\begin{equation}\label{fulldomcrit}
h\in\cA_{[c,f|_{\dom (M)},M]}\ \ \ \ \Leftrightarrow\ \ \ \ \ \alpha_{[c,f|_{\dom (M)},M]}\leq h\leq \gamma_{[c,f|_{\dom (M)},M]}.
\end{equation}
\end{theorem}

We see that besides existence and duality relations between optimal $c$-convex $c$-antiderivatives of the two dual families, an explicit construction of these optimal functions is provided in \eqref{gammafulldom} and \eqref{conjufulldom} in the case where the set where we keep the values of $f$ fixed is $S=\dom(M)$. An explicit construction of the optimal $c$-convex $c$-antiderivatives under the general assumptions of Theorem \ref{main} was also presented in \cite{BR2}. It is the second of the two main results mentioned above and is given in Theorem \ref{mainformula} below. To this end, we will need to recall the concept of $c$-cyclic monotonicity:

\begin{definition}[c-cyclic monotonicity]\label{cycmondef}
A mapping $M:X\rightrightarrows Y$ is said to be cyclically monotone of order $n$ with respect to $c$, $n$-$c$-monotone for short, when given any set of $n$ ordered pairs $\{(x_i,y_i)\}_{i=1}^n\subset G(M)$, if we set $x_{n+1}=x_1$, then
\begin{equation}\label{cycmonres}
0\leq\sum_{i=1}^n [c(x_i,y_i)-c(x_{i+1},y_i)].
\end{equation}
In this case we say that $G(M)$ is an $n$-$c$-monotone set. A mapping $M$ is said to be cyclically monotone with respect to $c$, $c$-cyclically monotone for short, if it is $n$-$c$-monotone for all $n\in\mathbb{N}$. 
\end{definition}

At the beginning of the next section we recall a characterization of $c$-cyclic monotonicity which lies at the core of optimal transport plans. Focusing on pure $c$-convexity theory results, we construct the minimal $c$-convex $c$-antiderivative $\alpha_{[c,f|_S,M]}$ in Theorem \ref{mainformula} below by also making use of the following well-known construction of an antiderivative due to Rockafellar. This fact from classical convex analysis also holds in the generality of $c$-convexity:

\begin{definition}[Rockafellar's antiderivative]
With the function $c$, the mapping $M:X\rightrightarrows Y$ and the point $s\in\dom (M)$, we associate Rockafellar's function $R_{[c,M,s]}:X\to\RX$, defined by
\begin{equation}
R_{[c,M,s]}(x):=
\sup_{\begin{array}{c}
                       n\in\mathbb{N},\\
                       x_1=s,\ x_{n+1}=x,\\
                       \{(x_i,y_i)\}_{i=1}^n\subset G(M)
                     \end{array} }
\ \ \sum_{i=1}^n [c(x_{i+1},y_i)-c(x_i,y_i)].
\end{equation}
\end{definition}

\begin{theorem}
A proper mapping $M:X\rightrightarrows Y$ is $c$-cyclically monotone if and only if it
has a proper $c$-antiderivative. In this case, in particular, for any $s\in\dom (M)$,
Rockafellar's function $R_{[c,M,s]}$ is a proper $c$-convex $c$-antiderivative
of $M$ which satisfies $R_{[c,M,s]}(s)=0$. In fact, $R_{[c,M,s]}$ is proper if and only if $M$ is proper and $c$-cyclically monotone.
\end{theorem}

Employing Rockafellar's antiderivative, we reestablish the nonemptiness of the family $\cA_{[c,f|_S,M]}$ by explicitly constructing the function $\alpha_{[c,f|_S,M]}$. Duality relations from Theorem \ref{main} now also yield an explicit construction of $\gamma_{[c,f|_S,M]}$, as we recall in an example in Section 4:

\begin{theorem}\label{mainformula}
Suppose that $f:X\to\RX$ is a $c$-antiderivative of the mapping $M:X\rightrightarrows Y$ and suppose that $\emptyset\neq S\subset\mathrm{dom} (M)$. Then the minimal $c$-antiderivative of $M$ that equals $f$ at the points of $S$, the function $\alpha_{[c,f|_S,M]}\in\cA_{[c,f|_S,M]}$, is given by
\begin{equation}\label{alpha}
\alpha_{[c,f|_S,M]}(x)=\ \sup_{s\in S}\ [f(s)+R_{[c,M,s]}(x)]\ \ \ \ \ \ \ \forall x\in X.
\end{equation}
\end{theorem}

\section{Optimal $c$-Antiderivatives as Optimal Prices Corresponding to an Optimal Transport Plan with Initial Price Constraints}

In order to reach the main discussion of this paper, we first discuss the underlying $c$-convexity structure of the Kantorovich duality theorem and an economic interpretation. Suppose we have mass produced at sources in $X$ and we have consumption targets for this mass in $Y$. The cost of transporting a unit of mass from the source $x\in X$ to the target $y\in Y$  is $c(x,y)$. The problem of optimal transport is to find where each unit of mass should go so that the total cost of transportation is minimal. The main result of the present paper, Theorem \ref{main opot} below, refines mainly the right-hand side of the Kantorovich duality, where we have $c$-convex functions as $c$-antiderivatives of the $c$-cyclically monotone set on which the optimal transport plan in concentrated. Before we focus on the right-hand side, we briefly focus solely on the left-hand side by pointing out that the fact that optimal transport plans are concentrated on $-c$-cyclically monotone sets is rooted in the following equivalent definition of $c$-cyclic monotonicity:    
  
\begin{proposition}\label{optimal transport cyclic equivalent}
Let $X$ and $Y$ be sets and let $c:X\times Y\to\RR$ be a function. Then the mapping  $T:X\rightrightarrows Y$ is $n$-$c$-cyclically monotone if and only if for any set of $n$ pairs $\{(x_i,y_i)\}_{i=1}^n\subset G(T)$, we have
\begin{equation}\label{cyclically monotone by permutations}
\sum_{i=1}^n c(x_i,y_{\sigma(i)})\leq\sum_{i=1}^n c(x_i,y_i)\ \ \ \ \ \ \ \forall\sigma\in S_n,
\end{equation}
where $S_n$ is the permutation group on $\{1,\ldots,n\}$. Consequently, the mapping $T$ is $c$-cyclically monotone  if and only if for every finite set of pairs $\{(x_i,y_i)\}_{i=1}^n\subset G(T)$, \eqref{cyclically monotone by permutations} is satisfied.
\end{proposition}
We see that a mapping is $c$-cyclically monotone if and only if for any finite set of pairs $\{x_i,y_i\}_{i=1}^n$ in its graph, the total cost of transporting unit masses from the sources $x_i$ to the targets $y_i$ is the highest if we transport according to the assignment of $T$, rather than according to any other rearrangement of transportation.
 
On the right-hand side of the Kantorovich duality we have $price$: suppose that a unit of mass is sold to the distributor (in charge of the transportation) at the source $x$ for the price of $f(x)$ and a unit of mass arriving at the target $y$ is sold by the distributor for the price of $g(y)$. Then, in order to make it worthwhile for the consumer at $y$ and the producer at $x$ to use the distributor and not deal directly between them, the distributor has to offer them a price difference which satisfies $g(y)-f(x)\leq c(x,y)$. Thus, the distributor has to come up with a pair of prices $(f,g)$ satisfying this inequality for all $x$ and $y$. Now, if the distributor buys from the producer at $x$ the mass $\mu(dx)$, then the price of buying it is $f(x)\mu(dx)$. We see that the total price when buying the mass is $\int_X f(x)\mu(dx)$ and, analogously, the total price when selling the mass is $\int_Y g(y)\nu(dy)$. When trading directly, the traders want to minimize the total cost of the transportation, which is the left-hand side of the Kantorovich duality. On the other side of the duality, the distributor seeks to maximize the profit, which is the difference between the total price of selling the mass and the total price of buying it. So far it is clear that if the distributor offers prices such that $g-f\leq c$, then the supremum over the total differences is not higher than the optimal total transport cost. It follows that if for a pair of prices $(f,g)$ such that $g-f\leq c$ and the plan $\pi$ we have equality, then $(f,g)$ is optimal in the dual Kantorovich problem and $\pi$ is an optimal plan. It is also clear that if indeed $-c(x,y)-(-g(y))\leq f(x)$, then supremizing over $y\in Y$ we get $-g^{-c}\leq f$. Equivalently, if $-c-f\leq -g$, then $f^{-c}\leq -g$. Thus, in order to maximize the difference between the total price of selling and buying, the distributor should look for a pair of prices $(f,g)$ such that $g=-f^{-c}$ and $f=(-g)^{-c}$. Thus, we might as well supremize over the pairs $(f,-f^{-c})$, where $f$ is $-c$-convex or over the pairs $((-g)^{-c},g)$, where $-g$ is $-c$-convex. Having a solution $(f,g)$ and a corresponding  optimal plan $\pi$, we see that $\pi$ is concentrated on the set of points $(x,y)$ where the equality $g(y)-f(x)=c(x,y)$ holds since otherwise, if $g(y)-f(x)<c(x,y)$ and a nonzero mass was transported from $x$ to $y$, then there is no equality in the duality for $(f,g)$ and $\pi$. Thus, $\pi$ is concentrated on the set where $g-f=c$, that is, where $f+f^{-c}=-c$, which is the $-c$-cyclically monotone set $G(\partial_{-c}f)$. In what our discussion covered so far in terms of the prices, no distinction between ``special'' optimizing prices was made. We now arrive at the main discussion of the present paper. We wish to introduce some additional market considerations into our discussion which will give rise to such a distinction. 

To this end, we consider the following monopolistic situation in the market of the mass in trade. First, we note the following assumption, underlying our entire discussion: the producers are assumed to be fair in the sense that they are impartial to the end consumer, that is, a producer at $x$ sells every unit of mass at a constant price, regardless of its destination $y$. Similarly, the end consumer is impartial to the producers. Now suppose that the mass already flows optimally with corresponding prices so that the duality is realized, and suppose further that the producers are united under the control of some body, perhaps all owned by a single owner, for example. The producers now wish to adjust prices. For example, because some long term supply contracts expired. On the other hand, a part of the optimal transport plan has to be kept intact, say due to still valid transport contracts. We let the mapping $M:X\rightrightarrows Y$ be the mapping such that $G(M)$ is the part of the optimal plan we keep fixed. Furthermore, some of the producers in  $\dom(M)$ will also have to keep their selling prices $f(x)$ fixed during the planned price adjustment, perhaps also due to still valid contracts with some of their customers. (This could be why these producers are in $\dom(M)$ to begin with.) Now, if the producers are to suggest a new optimal selling price (in the same sense as before, that is, it allows equality in the duality), while taking into account the fixed part $G(M)$ of the optimal transport plan and the set of fixed prices, which we denote by $S\subset\dom(M)$, then they will suggest the highest one possible. According to what we saw above, this price will then also have to be a $-c$-convex $-c$-antiderivative of $M$ and it should coincide with $f$ at the points of $S$ and therefore, according to Theorem \ref{main}, $\gamma_{[-c,f|_S,M]}$ will be the highest price possible. Since we still keep the fixed part of the plan intact, and also with the same price difference between the end points, the distributor should not care about these price changes. Furthermore, for the rest, we will then have an optimal plan which extends the fixed part and which is concentrated on the $-c$-cyclically monotone set $G(\partial\gamma_{[-c,f|_S,M]})$. However, taking all these considerations into account, the end customers would like the distributors' new buying price to be the lowest possible, which must then be $\alpha_{[-c,f|_S,M]}$. 

In the past, the problem of existence of optimal prices, sometimes compatible with conventional optimal transportation, sometimes in a different framework of transportation, drew the attention of authors in various special cases. In some of these cases, a construction of these optimizers was presented in $\RR^n$ under various specifications. A recent such example is \cite{butcar}. These types of problems are related to the economic notions of ``asymmetric information'', ``mechanism design'', ``incentive compatibility'', ``screening'', the ``principal-agent'' framework and others. An even more recent study of problems of optimal pricing, where one also finds a recollection and references of some of the past discussions, can be found in \cite{FHM}.  

At this point we proceed to our main results and explicit formulae under the full generality of the settings of the Kantorovich duality Theorem \ref{Kantorovich}. In order to make our discussion mathematically formal, we will employ Villani's monograph, again, as follows. In his discussion of restrictions of optimal transport plans, Villani presents the following two restriction results. The first result we quote is a pure $c$-convexity property.  

\begin{lemma}\label{Villani's lemma}
Let $X$ and $Y$ be two sets and let $c:X\times Y\to\RR$. Let $X'\subset X$, $Y'\subset Y$ and let $c'$ be the restriction of $c$ to $X'\times Y'$. Let $f:X\to\RX$ be a $c$-convex function. Then there is a $c'$-convex function $f':X\to\RX$  such that $f'\leq f$ on $X'$, $f'$ coincides with $f$ on the projection of $G(\partial_cf)\cap(X'\times Y')$ on $X$ and $G(\partial_c f)\cap(X'\times Y')\subset G(\partial_{c'}f')$.   
\end{lemma}
The above lemma lies at the heart of Villani's proof of the following restriction property of the Kantorovich duality theorem:

\begin{theorem}\label{Villani's restriction}
Let $(X,\mu)$ and $(Y,\nu)$ be two Polish probability measure spaces. Suppose that the cost function $c:X\times Y\to\RR$ is lower semicontinous and suppose further that it majorizes a function $a(x)+b(y),\ (x,y)\in X\times Y$, for some upper semicontinuous functions $a\in L_1(\mu)$ and $b\in L_1(\nu)$. Assume that the optimal total cost $C(\mu,\nu)$ is finite. Let $\pi$ be an optimal transport plan, and let $f$ be a $-c$-convex function such that $\pi$ is concentrated on $G(\partial_{-c}f)$. Let $\tilde{\pi}$ be a measure on $X\times Y$ satisfying $\tilde{\pi}\leq\pi$, and $\zeta=\tilde{\pi}[X\times Y]>0$; let $\pi':=\tilde{\pi}/\zeta$, and let $\mu'$ and $\nu'$ be the marginals of $\pi'$. Let $X'\subset X$ be a closed set containing the support of $\mu'$ and let $Y'\subset Y$ be a closed set containing the support of $\nu'$. Let $c'$ be the restriction of $c$ to $X'\times Y'$. Then there is a $-c'$-convex function $f':X'\to\RX$ such that \\
\\
(a) $f'$ coincides with $f$ on the projection of $G(\partial_{-c}f)\cap(X'\times Y')$ on $X$,  which has full $\mu'$-measure;\\
(b) $\pi'$ is concentrated on $G(\partial_{-c'}f')$;\\
(c) $f'$ solves the dual Kantorovich problem between $(X',\mu')$ and $(Y',\nu')$ with cost $c'$.
 
\end{theorem}
For the sake of convenience, and, mainly, for the sake of completeness of our argument below, we also quote Villani's proof:

\begin{proof}
Let $f'$ be defined by Lemma \ref{Villani's lemma}. To prove (a) it suffices to note that $\pi'$ is concentrated on        $G(\partial_{-c}f)\cap(X'\times Y')$, so $\mu'$ is concentrated on the projection of $G(\partial_{-c}f)\cap(X'\times Y')$ on $X$. Then $\pi$ is concentrated on $G(\partial_{-c}f)$, so $\tilde{\pi}$ is concentrated  on $G(\partial_{-c}f)\cap(X'\times Y')$, which by Lemma \ref{Villani's lemma} is contained in $G(\partial_{-c'}f')$; this proves (b). Finally, (c) follows from Theorem \ref{Kantorovich}.
\end{proof}

In order to refine Theorem \ref{Villani's restriction} and introduce optimal prices into our discussion, employing our results from Section 2, we first refine Villani's Lemma \ref{Villani's lemma} as follows. The fact that we do not assume the existence of a $c$-convex $c$-antiderivative $f$ in Theorem \ref{main} and Theorem \ref{mainformula} (we only assume the existence of a $c$-antiderivative) will be crucial for us. Indeed, $c$-convexity can be ruined when restricted; however, the property of being a $c$-antiderivative is retained through restriction.

\begin{lemma}\label{new restriction lemma}
Let $X$ and $Y$ be two sets and let $c:X\times Y\to\RR$. Let $X'\subset X$, $Y'\subset Y$ and let $c'$ be the restriction of $c$ to $X'\times Y'$. Let $f:X\to\RX$ be a $c$-convex function. Let $M:X'\rightrightarrows Y'$ be any mapping  such that 
$$
G(M)\subset (G(\partial_{c}f)\cap(X'\times Y'))
$$
and let $\emptyset\neq S\subset\dom(M)$. Then the set $\cA_{[c',f|_S,M]}$ of all $c'$-convex functions $f':X'\to\RX$ such that 
\begin{equation}\label{restriction constrain}
G(M)\subset G(\partial_{c'}f')\ \ \ \ \ \mathrm{and}\ \ \ \ f'|_S=f|_S,
\end{equation}
that is, $f'$ is a $c'$-antiderivative of $M$ and $f'|_S=f|_S$, is not empty. In particular, $\alpha_{[c',f|_S,M]}$ and $\gamma_{[c',f|_S,M]}$ are the minimal and the maximal $c'$-convex functions which satisfy \eqref{restriction constrain}, respectively. 
\end{lemma}

\begin{proof}
This will be seen to be a direct consequence of Theorem \ref{main} as soon as we provide a $c'$-antiderivative $\tilde{f}:X'\to\RX$ of $M$ such that $\tilde{f}|_S=f|_S$ (not necessarily $c'$-convex). Indeed, a straightforward verification, according to the definition of the $c'$-subdifferential, shows that the function $\tilde{f}:=f|_{X'}+\iota_{\dom (M)}$ will do the job. 
\end{proof}

We see that applying our results from Section 2, we recover Villani's Lemma \ref{Villani's lemma} as a particular case, namely, if we let $G(M):= (G(\partial_{c}f)\cap(X'\times Y'))$ and $\emptyset\neq S:=\dom(M)$, then the function $\alpha_{[c',f|_S,M]}$ has all the properties of the function $f'$ in the conclusion of Lemma \ref{Villani's lemma}.
   
Now, in Villani's proof of Theorem \ref{Villani's restriction} above the fact that $f'$ (which was supplied by the conclusion of Lemma \ref{Villani's lemma}) and $f$ coincide on the projection of $G(\partial_{c}f)\cap(X'\times Y')$ on $X$ was not employed. It was, in fact, part of the conclusion of Theorem \ref{Villani's restriction}. Dropping this conclusion (it will be replaced by the much more general outcome above), we see that the only properties of $f'$ that matter are that (1) it is $c'$-convex and that (2) $\pi'$ is concentrated on $ G(\partial_{c'}f')$. Thus, $f'$ can be replaced by any other function satisfying (1) and (2). Combining this discussion with our refined Lemma \ref{new restriction lemma}, we deduce the following theorem regarding optimal constrained prices $f'$. It is the main result of the present paper and embeds our discussion regarding optimal prices in the general mathematical framework of the Kantorovich duality Theorem \ref{Kantorovich}.

\begin{theorem}\label{main opot}
Assuming the hypotheses and settings of Theorem \ref{Villani's restriction}, let $M:X'\rightrightarrows Y'$ be any mapping such that $\pi'$ is concentrated on $G(M)$ and $G(M)\subset G(\partial_{-c}f)\cap(X'\times Y')$. Let $\emptyset\neq S\subset\dom{M}$. Then $\cA_{[-c',f|_S,M]}$ is the family of $-c'$-convex functions $f':X'\to\RX$ such that\\
\\
(a) $f'$ coincides with $f$ on $S$ and $\dom (f')$ has full $\mu'$ measure;\\
(b) $f'$ is a $-c'$-antiderivative of $M$ and, in particular, $\pi'$ is concentrated on $G(\partial_{-c'}f')$;\\
(c) $f'$ solves the dual Kantorovich problem between $(X',\mu')$ and $(Y',\nu')$ with cost $c'$.\\
\\
Furthermore, $\cA_{[-c',f|_S,M]}$ is not empty. In particular, $\alpha_{[-c',f|_S,M]}$ and $\gamma_{[-c',f|_S,M]}$ are the minimal and the maximal $c'$-convex functions with these properties, respectively. The function $\alpha_{[-c',f|_S,M]}$ is given, explicitly, by \eqref{alpha} and $\gamma_{[-c',f|_S,M]}$ is given via \eqref{conjugamma}. 
\end{theorem}

In the next section we discuss a family of examples where an additional structure is available due to the properties of the cost function we study, namely, a metric. Formula  \eqref{alpha} is then directly expanded in order to present $\alpha_{[c,f|_S,M]}$, and, indeed, as Theorem \ref{main opot} implies, we also expand $\gamma_{[c,f|_S,M]}$ via formula \eqref{conjugamma}. In general, the way to explicitly write an expansion of $\gamma_{[c,f|_S,M]}$ via formula \eqref{conjugamma} is to use formula \eqref{alpha} in order to write down the function $\alpha_{[c,f^c|_{M(S)},M^{-1}]}$ and then $c$-transforming it in order to get the function $\gamma_{[c,f^{cc}|_{(M^{-1}(M(S))},(M^{-1})^{-1}]}$ which is equal to the function $\gamma_{[c,f|_S,M]}$ (even if $S\varsubsetneq M^{-1}(M(S))$).

\begin{example}
Suppose that $Y=X,\ \mu=\nu$ and suppose that the cost of not moving a unit of mass is zero for every unit of mass, that is, $c(x,x)=0$ for every $x\in X$.
Suppose further that an optimal plan $\pi\in\Pi(\mu,\nu)$ is obtained by zero transportation, that is, $\pi$ is concentrated on the graph of the identity mapping $I:X\to X$. In this case, in each producer-consumer pair, the producer and his/her end consumer are at the same location in $X$, trading directly between them, with no middle distributor, so that in the dual Kantorovich problem the pairs of admissible prices are of the form $(f,g)=(f,f)$ and $f$ satisfies $f(y)-f(x)\leq c(x,y)$ for every $(x,y)\in X\times X$. Now, suppose the price $f$ can be adjusted while we keep the prices on a given set $S\subset X$ fixed and also preserve the direct trading in $S$. Then, employing the formulae from Theorem \ref{main}, the lowest and highest admissible prices possible are:
\begin{equation}\label{ex1}
\alpha_{[-c,f|_S,I_S]}(x)=\sup_{s\in S}\ [f(s)-c(x,s)]\ \ \ \ and\ \ \ \ \ \gamma_{[-c,f|_S,I_S]}(x)=\inf_{s\in S}\ [f(s)+c(x,s)],\ \ \ \ x\in X,
\end{equation}  
respectively. When $c$ is a metric on $X$, then formulae \eqref{ex1} are precisely the optimal McShane and Whitney (Lipschitz) extensions of $f$ which we will discuss more extensively in the next section.        
\end{example}

We proceed with the following simple example with non-zero optimal transportation: 

\begin{example}
Let $X=[-\frac{3}{2},-1]\cup[1,\frac{3}{2}],\ \ Y=\{-1,1\}$, let $\mu$ be the Lebesgue measure on $X$ and let $\nu$ be equally distributed on $Y$ between $-1$ and $1$, that is, $\nu=\frac{1}{2}\delta_{\{-1\}}+\frac{1}{2}\delta_{\{1\}}$. Suppose that the cost function $c:X\times Y\to\RR$ is the function $c(x,y)=x\cdot y$. Now, consider the transport plan $\pi\in\Pi(\mu,\nu)$, which is obtained by transporting the mass of the segment $[-\frac{3}{2},-1]$ to $\{1\}$ and the mass of the segment $[1,\frac{3}{2}]$ to $\{-1\}$. Then $\pi$ is concentrated on the graph of the $-c$-cyclically monotone function $M:X\to Y$, defined by
$$
M(x)=\Big\{\begin{array}{cc}
\ \ 1 & \ \ -\frac{3}{2}\leq x\leq -1; \\ 
-1 & 1\leq x\leq\frac{3}{2}
\end{array} 
$$ 
and we have
$$
\int_{X\times Y} c(x,y)d\pi(x,y)=-\frac{5}{4}.
$$  
Consider also the price $f:X\to\RR$ defined by $f(x)=|x|$ and the zero price $g:Y\to\RR$. Then 
\begin{equation}\label{total price difference}
g(y)-f(x)\leq c(x,y)\ \ \ \forall(x,y)\in X\times Y\ \ \ \ and\ \ \ \int_Y g(y)d\nu(y)-\int_X f(x)d\mu(x)=-\frac{5}{4}.
\end{equation}
Since we have equality in the Kantorovich duality, $\pi$ is an optimal transport plan and the pair $(f,g)$ solves the dual Kantorovich problem. Suppose now that we keep the price $f$ fixed on the set $S=[1,\frac{3}{2}]$ and we look for prices on all of $X$, which are compatible with the optimal transport plan $\pi$. Then the set of all such prices is 
\begin{equation}\label{example formula of A}
\cA_{[-c,f|_S,M]}=\Bigg\{h_p:X\to\RR\ \ \Big|\ \ h_p(x):=\Big\{\begin{array}{cc}
-x+p &\ \  -\frac{3}{2}\leq x\leq -1; \\ 
x & 1\leq x\leq\frac{3}{2}
\end{array},
\ \ -2\leq p\leq 2\ \Bigg\},     
\end{equation}
where its envelopes, the optimal compatible prices $\alpha_{[-c,f|_S,M]}$ and $\gamma_{[-c,f|_S,M]}$, are obtained by letting $p=-2$ and $p=2$ in the formula in \eqref{example formula of A}, respectively. If $h_{p}\in\cA_{[-c,f|_S,M]}$ for $-2\leq p\leq 2$, then 
\begin{equation}\label{example dual formula}
h_{p}^{-c}(y)=\Big\{\begin{array}{cc}
0 &\  y=-1 \\ 
-p & y=1,
\end{array} 
\end{equation}   
which defines the $-c$-dual family $\cA_{[-c,f^{-c}|_{M(S)},M^{-1}]}$. In particular, its envelopes $\alpha_{[-c,f^{-c}|_{M(S)},M^{-1}]}$ and $\gamma_{[-c,f^{-c}|_{M(S)},M^{-1}]}$ are obtained by letting $p=2$ and $p=-2$ in \eqref{example dual formula}, respectively. It follows that the pairs of prices which solve the dual Kantrovich problem with the price constraints $f$ and $-f^{-c}$ on the sets $S\subset X$ and $M(S)\subset Y$, respectively, are of the form $(h,-h^{-c}),\ h\in\cA_{[-c,f|_S,M]}$. Indeed, for $-2\leq p\leq 2$, we have $-h_p^{-c}(y)-h_p(x)\leq c(x,y)$ for all $(x,y)\in X\times Y$ and the same total price difference as in \eqref{total price difference}. If, for $p\notin [-2,2]$, we let $f'$ be defined by the formula in \eqref{example formula of A} and $g'$ be defined by \eqref{example dual formula}, then we would still have the same total price difference. However, the pair $(f',g')$ is not a solution to the dual Kantorovich problem since the price difference constraint $g-f\leq c$ is not satisfied. Also, in this case, $f'$ is not a $-c$-antiderivative of $M$. 

\end{example}

Finally, our discussion in this section clearly raises the issue of uniqueness of solutions to the dual Kantorovich problem. Much is to be said in this context; see, for example, \cite{vil}. In the context of our discussion here, it is natural to ask under which conditions the family $\cA$ is a singleton. At this point we relegate this question to a future study.

\section{ Example: an Additional Perspective on Optimal Prices in the Metric Case}

When the cost function is a metric, then the Knatorovich duality theorem becomes the well-known Kantorovich-Rubinstein formula for Lipschitz functions. As an example of optimal prices on the right-hand side of the duality, we now wish to draw the reader's attention to a main result from \cite{BR2}. The following discussion allows us to drop all measure theoretic issues since the discussion is valid on a more basic level of sets which does not require measurability. In this case we also benefit from an additional interpretation of optimal prices as optimal constrained Lipschitz extensions which, in its simplest form, recovers the well-known optimal Lipschitz extensions of McShane \cite{mcs} and Whitney \cite{whi}:

\begin{theorem}[McShane, Whitney]
Let $S$ be a nonempy subset of a metric space $(X,d)$ and let $f:S\to\mathbb{R}$ be $1$-Lipschitz. Then $f$ extends to a $1$-Lipschitz function which is defined on all of $X$. In particular, the functions
$$
\alpha(x)=\sup_{s\in S}\ [f(s)-d(x,s)]\ \ \ \ \ \ \ \ \ \mathrm{and}\ \ \ \ \ \ \ \ \ \gamma(x)=\inf_{s\in S}\ [f(s)+d(x,s)]
$$
are $1$-Lipschitz extensions of $f$. If $h:X\to\mathbb{R}$ is $1$-Lipschitz and $h|_S=f$, then $\alpha\leq h\leq\gamma$.
\end{theorem}

In our discussion of Lipschitz functions we will assume that the functions are $1$-Lipschitz. Since $Kd^\alpha$, where $0<\alpha\leq 1$ and $K>0$, is also a metric whenever $d$ is, our results and formulae are also easily extensible to $\alpha$-H\"{o}lder continuous functions with constant $K$ by replacing $d$ with $Kd^\alpha$. In what follows below we present, in a more detailed and explicit form, an extension result of the following nature:

\begin{theorem}\label{lipext1}
Let $(X,d)$ be a metric space. Let $f:A\subset X\to\RR$ be a $1$-Lipschitz function and let $M:A\rightrightarrows A$ be a mapping such that
\begin{equation}\label{cond2}
f(y)-f(x)=d(x,y)\ \ \ \ \ \ \mathrm{for\ all}\ \ (x,y)\in G(M).
\end{equation}
Given $\emptyset\neq S\subset\dom (M)$, $f|_S$ extends to a $1$-Lipschitz function which is defined on all of $X$ and which satisfies (\ref{cond2}). Moreover, there exist minimal and maximal $1$-Lipschitz extensions of $f|_S$ which are defined on $X$ and satisfy (\ref{cond2}). The family of all such extensions is  $\cA_{[-d,f|_S,M]}$ and the optimal extensions are $\alpha_{[-d,f|_S,M]}$ and $\gamma_{[-d,f|_S,M]}$, respectively. 
\end{theorem}

The proof in \cite{BR2} of the above result does not rely on the McShane and Whitney extensions and these follow as the most particular case, as presented in more detail below. We see that $M$ is the part of the optimal plan which we keep fixed and $S$ is the set of points with fixed prices. However, as we remarked above, we will not use these interpretations and will just treat these as constraints for our Lipschitz extensions to comply with. The underlying nature of $-d$-convexity that allows our discussion of Lipschitz functions is rooted in the following equivalences, which can be found along with a proof and some historical remarks in \cite{BR2}.

\begin{proposition}\label{lipconv}
Let $(X,d)$ be a metric space. For any proper function $f:X\to\RX$, the following assertions are equivalent:\\
\\
(1) $f$ is $1$-Lipschitz;\\
(2) $f^{-d}=-f$;\\
(3) $f$ is $-d$-convex;\\
(4) $f$ is a $-d$-antiderivative of the identity $I:X\to X$. That is, $G(I)\subset G(\partial_{-d} f)$.\\
\\
In this case, it follows that $f:X\to\RR$. The graph of the $-d$-subdifferential of $f$ is the set of ordered pairs $(x,y)\in X\times X$ such that $f$ preserves the distance $d(x,y)$ in the sense that $f(y)-f(x)=d(x,y)$.
\end{proposition}

The proof of Proposition \ref{lipconv} does not rely on the fact that $d(x,y)=0$ only if $x=y$. Thus, Proposition \ref{lipconv} also holds for a pseudometric $d$. The fact that the identity mapping is $-d$-cyclically monotone is an immediate consequence of the properties of the metric $d$. It follows that the identity mapping is the most trivial $-d$-subdifferential of $-d$-convex functions. Having Proposition \ref{lipconv} at hand and combining it with Theorems \ref{main} and \ref{mainformula}, we were able to prove in \cite{BR2} the following main example of optimal $c$-convex $c$-antiderivatives. To this end, we reformulate the hypotheses of Theorem \ref{lipext1}. Both in Theorem \ref{lipext1} and in Theorem \ref{Lipext}, the hypotheses provide us with a $-d$ antiderivative $f$ of the mapping $M:X\rightrightarrows X$ and a nonempty subset $S$ of $\dom (M)$. Therefore the existence of the optimal extensions of $f|_S$ which are $-d$-convex $-d$-antiderivatives of $M$ and their formulae hold in both cases. Unlike in Theorem \ref{lipext1}, in Theorem \ref{Lipext} the reformulation does not require $f$ to be $1$-Lipschitz outside $\dom (M)$. However, it is clear that any $1$-Lipschitz extension of $f|_S$ that satisfies these hypotheses is also uniquely determined on $M(S)$ by the equality
\begin{align}
f(t)-f(s)=d(s,t)\ \ \ \  \mathrm{for\ every}\ s\in S\ \ \mathrm{and}\ \ t\in M(s).\nonumber\\
\nonumber
\end{align}

\begin{theorem}\label{Lipext}
Let $(X,d)$ be a metric space. Let $M:\dom(M)\subset X\rightrightarrows X$ and $f:\dom(M)\to\RR$ satisfy
\begin{equation}\label{dsubM}
f(x)-f(x')\leq d(x',y)-d(x,y)\ \ \ \ for\ all\ \ (x,y)\in G(M)\ \ and\ \ x'\in\dom(M).
\end{equation}
Let $\emptyset\neq S\subset\dom(M)$. Then $f|_S$ extends to a $1$-Lipschitz function $h:X\to\RR$ which satisfies
\begin{equation}\label{dsub}
h(x)-h(x')\leq d(x',y)-d(x,y)\ \ \ \ for\ all\ \ (x,y)\in G(M)\ \ and\ \ x'\in X.
\end{equation}
The set $\cA_{[-d,f|_S,M]}$ of all such extensions $h$ of $f|_S$ is convex. In particular, the function $\alpha_{[-d,f|_S,M]}:X\to\RR$ defined by
\begin{equation}\label{minfSM}
\alpha_{[-d,f|_S,M]}(x)=\sup_{\begin{array}{c}
                       s\in S,\\
                       n\in\mathbb{N},\ x_1=s,\\
                       \{(x_i,y_i)\}_{i=1}^n\subset G(M)\\
                     \end{array} }
f(s)+\sum_{i=1}^{n-1}[ d(x_i,y_i)-d(x_{i+1},y_i)]+d( x_n,y_n)-d(x,y_n)
\end{equation}
is the minimal $1$-Lipschitz function that agrees with $f$ on $S$ and satisfies \eqref{dsub}. The function $\gamma_{[-d,f|_S,M]}:X\to\RR$ defined by
\begin{equation}\label{maxfSM}
\gamma_{[-d,f|_S,M]}(x)=\inf_{\begin{array}{c}
                       s\in S,\\
                       n\in\mathbb{N},\ x_1=s,\\
                       \{(x_i,y_i)\}_{i=1}^n\subset G(M)\\
                     \end{array} }
f(s)+\sum_{i=1}^{n-1} [d(x_i,y_{i+1})-d(x_{i+1},y_{i+1})]+d(x_n,x)
\end{equation}
is the maximal $1$-Lipschitz function that agrees with $f$ on $S$ and satisfies \eqref{dsub}. If $S=\dom(M)$, then
\begin{equation}\label{minfM}
\alpha_{[-d,f|_{\mathrm{dom}(M)},M]}(x)=\sup_{(s,t)\in G(M)}\ [f(s)+d(s,t)-d(x,t)]
\end{equation}
and
\begin{equation}\label{maxfM}
\gamma_{[-d,f|_{\mathrm{dom}(M)},M]}(x)=\inf_{s\in\dom (M)}\ [f(s)+d(x,s)],\ \ \ \ x\in X.
\end{equation}
In particular, suppose that $f:S\subset X\to X$ is $1$-Lipschitz. Then $M=I_S$, where $I_S$ is the identity mapping on $S$,  satisfies \eqref{dsubM}. Consequently $\mathrm{(McShane,\ Whitney)}$,
\begin{equation}\label{minmaxfSI}
\alpha_{[-d,f|_S,I_S]}(x)=\sup_{s\in S}\ [f(s)-d(x,s)]\ \ \ \ \ \ \ and\ \ \ \ \ \ \ \gamma_{[-d,f|_S,I_S]}(x)=\inf_{s\in S}\ [f(s)+d(x,s)]
\end{equation}
are the minimal and maximal $1$-Lipschitz extensions of $f$, respectively.
\end{theorem}

For more details regarding this example, the reader is, once again, referred to \cite{BR2}.

\section*{Acknowledgments}
This research was supported in part by the Israel Science Foundation (Grant 389/12), the Fund for the Promotion of Research at the Technion
 and by the Technion General Research Fund.
\ \\
\ \\

\end{document}